\documentclass[a4paper,12pt]{amsart}
\usepackage[utf8]{inputenc}
\usepackage[T1]{fontenc}
\usepackage[english]{babel}
\usepackage{amssymb,amsthm,times}
\usepackage{xcolor}
\usepackage{hyperref}
\usepackage{cancel}
\usepackage{graphicx}
\usepackage{float}
\usepackage{epstopdf}

\newtheorem{TEOR}{Theorem}[section]
\newtheorem{PROP}[TEOR]{Proposition}
\newtheorem{CORO}[TEOR]{Corollary}

\newtheorem{DEFI}[TEOR]{Definition}

\newtheorem{EJEM}[TEOR]{Example}

\newtheorem{REM}[TEOR]{Remark}

\newcommand{\dom}{\operatorname{dom}}
\newcommand{\NWD}{\operatorname{NWD}}
\newcommand{\KK}{\operatorname{\mathcal{K}}}
\newcommand{\Fin}{\operatorname{Fin}}

\def\N{\mathbb{ N}}

\def\rto{\rightarrow}

\def\FF{\mathcal{F}}
\def\BB{\mathcal{B}}
\def\AA{\mathcal{A}}

\def\II{\mathcal{I}}
\def\JJ{\mathcal{J}}

\def\diam{\text{diam}}

\title{\bf Baire theorem for ideals of sets }
\author{A. Avil\'es}
\address{Departamento de Matem\'{a}ticas, Universidad de Murcia,
30100 Espinardo (Murcia), Spain} \email{avileslo@um.es}

\author{V. Kadets}
\address{Department of Mathematics and Informatics, Kharkiv V.N.Karazin National University,
pl. Svobody 4, 61022 Kharkiv, Ukraine} \email{v.kateds@karazin.ua}

\author{A. P\'erez}
\address{Departamento de Matem\'{a}ticas, Universidad de Murcia,
30100 Espinardo (Murcia), Spain} \email{antonio.perez7@um.es}

\author{S. Solecki}
\address{Department of Mathematics, University of Illinois at Urbana-Champaign}
\email{ssolecki@math.uiuc.edu}

\thanks{The first author was supported by MINECO and FEDER (MTM2014-54182) and by Fundaci\'{o}n S\'{e}neca - Regi\'{o}n de Murcia (19275/PI/14). The research of the second author partially was done during his stay in Murcia under the support of MEC and FEDER projects MTM2008-05396 and MTM2011-25377. The third author was partially supported by the  
MINECO/FEDER project  {MTM2014-57838-C2-1-P} and a PhD fellowship of ``La Caixa Foundation''. The fourth author was supported by NSF grant DMS-1266189.
}

\subjclass[2010]{54E52; 28A05; 54H05; 06A07}

\keywords{Baire theorem; ideal of sets; nowhere dense; analytic set}

\begin{document}
\maketitle

\begin{abstract}
We study ideals $\II$ on $\N$ satisfying the following Baire-type property: \emph{if $X$ is a complete metric space and $\{X_{A} \colon A \in \II \}$ is a family of nowhere dense subsets of $X$ with $X_{A} \subset X_{B}$ whenever $A \subset B$, then $\bigcup_{A \in \II}{X_{A}} \neq X$}. We give several characterizations and determine the ideals having this property among certain classes like analytic ideals and P-ideals. We also discuss similar covering properties when considering  families of compact and meager subsets of $X$.
\end{abstract}

\section{Introduction}
For a given set $X$ we denote as usual by $\mathcal{P}(X)$ the collection of all subsets of $X$. We call a set $\II \subset \mathcal{P}(\N)$ an \emph{ideal} if $\N \notin \II$ and given $A, B \in \II$ we have that $\mathcal{P}(A) \subset \II$ and $A \cup B \in \II$. A set $\beta \subset \II$ is a basis of $\II$ if every $A \in \II$ is contained in some $B \in \beta$. The \emph{character} of $\II$ is the minimal cardinality of a basis of $\II$. Along this paper, every considered ideal $\II$ is supposed to contain the ideal $\Fin$ of all finite subsets of $\N$.  

As a dual concept, a set $\FF \subset \mathcal{P}(\N)$ is a \emph{filter} on $\N$ if $\{ \N \setminus A\colon A \in \FF \}$ is an ideal on $\N$, and $\beta \subset \FF$ is called a basis of $\FF$ if every $A \in \FF$ contains some $B \in \beta$. If $(x_{n})_{n \in \N}$ is a sequence in a topological space $X$, then it is said to be \emph{$\mathcal{F}$-convergent} to $a \in X$, usually written $a =\lim_{n, \FF}{x_{n}}$, if for every neighbourhood $V$ of $a$ we have that $\{ n \in \N\colon x_{n} \in V \}$ belongs to $\FF$.

Let $\FF$ be a filter on $\N$ and let $E$ be an arbitrary Banach space. A sequence $(e_{n})_{n \in \N}$ in $E$ is said to be an \emph{$\FF$-basis} of $E$ if for every $x \in E$ there exists a unique sequence of scalars $(a_{n})_{n \in \N}$ such that
\[ x = \lim_{n, \FF}{\sum_{i=1}^{n}{a_{i} e_{i}}}. \]
This definition extends the notion of Schauder basis, which corresponds to the case $\FF_{cf} := \{ A \subset \N \colon \N \setminus A \in \Fin \}$, known as the Fr\'{e}chet filter. The concept of $\FF$-basis was introduced in \cite{Gan-kad}, but previously considered in \cite{con-gan-kad} for the filter of statistical convergence 
\[ \FF_{st}:= \left\{ A \subset \N\colon \lim_{n}{\frac{|A \cap \{ 1, \ldots, n\}|}{n}}=1 \right\}. \] 
It is clear from the definition of $\FF$-basis that the coefficient maps $e_{n}^{\ast}(x) = a_{n}$ are linear on $E$. However, and in contrast with Schauder bases, it is not known whether the $e_{n}^{\ast}$'s are necessarily continuous. A partial result was given by T. Kochanek \cite{kochanekFbases}, who showed that if $\FF$ has character less than $\mathfrak{p}$ then the answer is positive. Here $\mathfrak{p}$ denotes the \emph{pseudointersection number}, defined as the minimum of the cardinals $\kappa$ for which the following claim is true: \emph{if $\AA$ is a family of subsets of $\N$ with cardinality less than $\kappa$ and satisfying that $\bigcap{\AA_{0}}$ is infinite for each finite subset $\AA_{0} \subset \AA$, then there is an infinite set $B \subset \N$ such that $B \setminus A$ is finite for every $A \in \AA$}.

If we work with the dual ideal $\II$ associated to $\FF$, a review of Kochanek's argument shows that the key step to get the result is that $\II$ has the next property:
\begin{enumerate}
\item[($\Box$)]\label{propertyBox} If $X$ is a complete metric space and $\{ X_{A}\colon A \in \II\}$ is a set of meager subsets of $X$ with $X_{A} \subset X_{B}$ whenever $A \subset B$, then $ \bigcup{\{ X_{A}\colon A \in \II \}} \neq X$. 
\end{enumerate}
Unfortunately not every ideal has this property. If the character of $\II$ is less than $\mathfrak{p}$, then it has property ($\Box$), since $\bigcup{\{ X_{A}: A \in \II \}}$ is equal to $\bigcup{\{ X_{B}: B \in \beta \}}$ which is the union of less than $\mathfrak{p}$ meager subsets, and this is again a meager subset of $X$ by \cite[Corollary 22C]{fremlinMartin}. In section \ref{sec:baireIdeals}, we show that the converse is also true under the set-theoretical assumption $\mathfrak{p} = \mathfrak{c}$.

The aim of this paper is to study what happens if we replace the condition ``meager'' by ``nowhere dense'' in ($\Box$). Ideals satisfying this last property will be called \emph{Baire ideals}. In section \ref{sec:characterizations} we prove several characterizations of this type of ideals. We also show that in order to demonstrate that an ideal $\II$ is a Baire ideal, we just have to check property ($\Box$) (with nowhere dense subsets instead of meager ones) for the metrizable space $X = D^{\N}$, $D$  being the discrete space of cardinality equal to $\mathfrak{c}$.

The fourth section is devoted to determine which are the Baire ideals in the classes of analytic ideals and P-ideals. Here we work in ZFC without any other set-theoretical assumptions. We show that in both cases the only Baire ideals are the countably generated ideals. In contrast with this, we construct in section \ref{sec:example} a model of ZFC in which we can find an uncountably generated P-ideal $\II$ satisfying property ($\Box$) for the particular case $X = 2^{\N}$. We also study the case of ideals generated by an almost disjoint family of subsets of $\N$.

In the last section, we show that if in ($\Box$) one considers compact subsets instead of meager ones, then there are uncountably generated $F_{\sigma}$ ideals satisfying that property for $X = \N^{\N}$.

Our notation and terminology is standard and it is either
explained when needed or can be found in \cite{kechris-descriptive} and \cite{kelleyTopology}.

\section{Baire ideals}
\label{sec:baireIdeals}

As we announced in the introduction, if we assume that $\mathfrak{p} = \mathfrak{c}$ (for instance, under Martin's Axiom), property ($\Box$) depends exclusively on the character of the ideal $\II$, as the following proposition shows.

\begin{PROP}\label{PROP:charactDiamond}
If $\II$ has a basis of cardinality less than $\mathfrak{p}$ then it has property ($\Box$). If $\mathfrak{p} = \mathfrak{c}$, then the converse is true.
\end{PROP}
\begin{proof}
The first part was discussed in the introduction so we just prove the second statement. Suppose that $\II$ is an ideal without a basis of cardinality less than $\mathfrak{p}$. Then we can easily construct a basis $\{ B_{ \alpha}\colon \alpha \in \mathfrak{p} \}$ with the property that $B_{\gamma}$ does not belong to the ideal generated by $\{ B_{\alpha}\colon \alpha < \gamma \}$. Consider the map $g(\cdot)\colon \II \rightarrow \mathfrak{p}$ that associates to each $A \in \II$ the minimum of the ordinals $\gamma$ such that $A$ belongs to the ideal generated by $\{ B_{\alpha}\colon \alpha < \gamma \}$. 
Now we can enumerate the elements of $\mathbb{R}$ as $\{ x_{\alpha}\colon \alpha \in \mathfrak{p}\}$ and define for each $A \in \II$ the set $\mathbb{R}_{A}:= \{ x_{\alpha}\colon \alpha \leq g(A)\}$. This is a meager set (union of less than $\mathfrak{p}$ meager sets is meager by \cite[Corollary 22C]{fremlinMartin}), and obviously $\mathbb{R}_{A} \subset \mathbb{R}_{B}$ whenever $A \subset B$. But it is also clear that $\bigcup{\{ \mathbb{R}_{A}\colon A \in \II \}} = \mathbb{R}$ since $x_{\alpha} \in \mathbb{R}_{B_{\alpha}}$ for every $\alpha < \mathfrak{p}$.  
\end{proof}

It is natural to ask whether the situation is the same if we restrict to nowhere dense subsets instead of meager ones. Given a topological space $X$, we will denote by $\NWD(X)$ the family of all nowhere dense subsets of $X$ ordered with inclusion.\label{DEFI:familiesOfSubsets}

\begin{DEFI}
\label{DEFI:baireIdeal}
Let $\II$ be an ideal on $\N$ and $X$ a topological space. We call $\II$ a \emph{Baire ideal for $X$} if for every monotonic function $f\colon \II \rightarrow \NWD(X)$ we have that
\[ \bigcup{\{ f(A)\colon A \in \II \}} \neq X. \]
If $\II$ has this property for every complete metric space $X$ then we simply say that $\II$ is a \emph{Baire ideal}.
\end{DEFI}

It follows from Proposition \ref{PROP:charactDiamond} that every ideal $\II$ with character less than $\mathfrak{p}$ is a Baire ideal. However, it is not clear if the converse is true, even under Martin's Axiom. It is interesting to remark that if we put less restrictions on $X$ then we can give a full characterization.

\begin{PROP}\label{PROP:II-cat spaces}
Let $\II$ be an ideal on $\N$. Then, $\II$ is countably generated if and only if it is a Baire ideal for every topological space $X$ of second category in itself.
\end{PROP}

\begin{proof}
If $X$ is of second category in itself then it is not the union of a countable family of nowhere dense subsets. Hence, by the comments preceding this Proposition, if $\II$ is countably generated then it is a Baire ideal for $X$. 

To see the converse, suppose that $\II$ is not countably generated. We can endow $X=\II$ with the topology $\tau$ generated by the basis of closed sets consisting of $\emptyset$ and the sets
\[ X_{A} = \{ I \in \II\colon I \subset A \} \: (A \in \II). \]
Every nowhere dense subset of $(X, \tau)$ is contained in some $X_{A}$, and moreover every $X_{A}$ is nowhere dense since if a basic open $X \setminus X_{B}$ was contained in $X_{A}$, then we would have that $X = X_{A} \cup X_{B}$, contradicting that $\II$ is not countably generated. The same fact shows that $X \neq \bigcup_{n \in \N}{X_{A_{n}}}$ for every sequence $(A_{n})_{n \in \N}$ in $X$, so $X$ is not the union of a countable family of nowhere dense subsets and hence it is of second category in itself. Finally it is clear that $\bigcup{\{ X_{A}\colon A \in \II \}} = X$ and so $\II$ is not a Baire ideal for $X$.
\end{proof}

\section{Characterizations of Baire ideals}
\label{sec:characterizations}

Recall that the \emph{weight} $w(X)$ of a topological space is the minimum cardinal $\kappa$ for which $X$ has basis of open sets with such cardinality. In the current section, we will show that in order to prove that $\II$ is a Baire ideal for every complete metric space $X$ of weight $w(X) \leq \kappa$ we just have to check that it is a Baire ideal for $D^{\N}$ where $D$ is a discrete space of cardinality $\kappa$. Moreover, we give another characterization of this fact in terms of the properties of certain subtrees of $\mathcal{P}(\II)$.

If $D$ is a discrete space then $D^{ \N}$ with the product topology is completely metrizable. We introduce some notation: If $\sigma\colon \N \rto D$ is an element of $D^{\N}$ then we put $\sigma|_{0} := \emptyset$ and $\sigma|_{k} = (\sigma(1),...,\sigma(k))$ if $k > 0$. We denote by $D^{< \N}$ the set of finite sequences of elements in $D$. In other words, $t \in D^{<\N}$ if there is $\sigma \in D^{\N}$ and $k \geq 0$ such that $\sigma|_{k} = t$. In this case we will write $t \preccurlyeq \sigma$. If $s \in D^{< \N}$ we write $s \preccurlyeq t$ if $t$ is an extension of $s$; i.e. if there are $\sigma \in D^{\N}$ and $0 \leq p \leq q$ with $s=\sigma|_{p}, t = \sigma|_{q}$. If $\alpha \in D$ then we denote by $s \textasciicircum \alpha := (s(1),...,s(k),\alpha)$ the extended sequence obtained by adding the element $\alpha$ at the end of $s$. With this notation, a basis of open subsets in $D^{\N}$ is given by the sets
\[ U_{s} = \{ \sigma \in D^{\N}\colon \sigma \succcurlyeq s \} \hspace{5mm} (s \in D^{< \N}).\]
Notice that each $U_{s}$ is clopen. Recall that $\AA \subset \mathcal{P}(\N)$ is called \emph{hereditary} if $B \subset A \in \AA$ implies that $B \in \AA$.

\begin{TEOR}
\label{TEOR:treeWeight}
Let $D$ be a discrete space of (infinite) cardinality $\kappa$ and $\II$ an ideal on $\N$. The following assertions are equivalent:
\begin{enumerate}
\item[$(1)$] $\II$ is a Baire ideal for every complete metric space $X$ with $w(X) \leq \kappa$.
\item[$(2)$] $\II$ is $D^{\N}$-Baire.
\item[$(3)$] Let $f\colon D^{< \N} \rto \mathcal{P}(\II)$ be monotonic and such that for every $s \in D^{<\N}$: 

\noindent (i) $f(s)$ is hereditary, (ii) $\II = \bigcup_{t \succcurlyeq s}{f(t)}$.

\noindent Then, there exists $\sigma \in D^{\N}$ such that $\bigcup_{k \in \N}{f(\sigma|_{k})} = \II$. 

\item[$(4)$] Let $f\colon D^{< \N} \rto \mathcal{P}(\II)$ be monotonic and such that for every $s \in D^{<\N}$: 

\noindent (i) $f(s)$ is hereditary, (ii') $\II=\bigcup_{d \in D}f(s \textasciicircum d)$, (iii) $f(s) = \bigcap_{d \in D}{f(s \textasciicircum d)}$. 

\noindent Then, there exists $\sigma \in D^{\N}$ such that $\bigcup_{k \in \N}{f(\sigma|_{k})} = \II$.
\end{enumerate}
\end{TEOR}

\begin{proof}
Implications $(1)\Rightarrow (2)$ and $(3) \Rightarrow (4)$ are obvious.

\noindent $(2) \Rightarrow (3)$: Suppose that we have a function $f\colon D^{< \N} \rto \mathcal{P}(\II)$ as in $(3)$ and define the map $F\colon \II \rto \NWD(D^{\N})$ that assigns to each $A \in \II$ the set
\[ F(A) = \{ \sigma \in D^{\N}\colon A \notin f(\sigma|_{k}) \text{ for every $k \in \N$} \}. \]
To see that $F(A)$ is effectively nowhere dense, notice that if $U_{s}$ is a basic open then there exists by (ii) an element $t \succcurlyeq s$ with $A \in f(t)$, which means that $U_{t} \subset U_{s} \setminus F(A)$. Since we are assuming that $\II$ is a Baire ideal for $D^{\N}$, we can find $\sigma \in D^{\N} \setminus \bigcup{\{ F(A)\colon A \in \II \}}$.  But $\sigma \notin F(A)$ means that $A \in \bigcup_{k \in \N}{f(\sigma|_{k})}$ for every $A \in \II$. 

\noindent $(4)\Rightarrow(1)$: Let $\{ W_{d}\colon d \in D \}$ be a basis of open sets in $X$ (repeating elements if necessary). We are going to construct a collection $\{ V_{s}\colon s \in D^{< \N}\}$ of open sets in the following inductive way: $V_{\emptyset} = X$ and $V_{(d)}:=W_{d}$ for each $d \in D$. Suppose that we have constructed $V_{s}$ ($s \in D^{< \N}$). The elements $V_{s \textasciicircum x}$ $(x \in D)$ are going to be all the open sets $W_{d}$ such that $\overline{W_{d}} \subset V_{s}$ and $\diam(W_{d}) \leq \diam(V_{s})/2$ repeating elements if necessary (here $\diam{(C)}$ means the diameter of the set $C$). Observe that for every $\sigma \in D^\N$, the branch $\{ V_{\sigma|_{k}} \colon k \in \N \}$ has non-empty intersection which consists of a unique element that we will denote by $p_{\sigma}$. Of course different branches could determine the same point, and it is also clear that every element of $X$ belongs to the intersection of one of these branches. Moreover, each open set $V_{s}$ can be described now as $V_{s} = \{ p_{\sigma}\colon \sigma \succcurlyeq s \}$.

\noindent Given $F\colon \II \rto \NWD(X)$ monotonic, define a new map $f\colon D^{< \N} \rto \mathcal{P}(\II)$ as
\[ f(s) = \{ A \in \II\colon \overline{F(A)} \cap V_{s} = \emptyset \} \mbox{ for each $s \in D^{< \N}$}. \]
It is also monotonic and satisfies condition (i). To see (ii'), fix $s \in D^{< \N}$ and take an arbitrary $A \in \II$. Since $F(A)$ is nowhere dense, there exists an open set $W_{d}$ such that $\overline{W_{d}} \subset V_{s} \setminus F(A)$ and $\diam(W_{d}) \leq \diam(V_{s})/2$. We know that this $W_{d}$ is equal to some element $V_{s \textasciicircum x}$, so $A \in f(s \textasciicircum x)$. We check now condition (iii) on $f$. If $s \in D^{< \N}$ and $A \notin f(s)$ then $F(A) \cap V_{s} \neq \emptyset$, so there exists an element $p_{\sigma}$ ($\sigma \succcurlyeq s$) that belongs to this intersection. If $\sigma|_{k} = s$, then notice that $V_{\sigma|_{k+1}} = V_{s \textasciicircum \sigma(k+1)}$ also has non-empty intersection with $F(A)$, so $A \notin f(\sigma|_{k+1})$. This shows that $f(s) \supset \bigcap_{d \in D}{f(s \textasciicircum d)}$, but this is indeed an equality by monotonicity.

Using the assumption, there is $\sigma \in D^{\N}$ such that every $A \in \II$ belongs to $f(\sigma|_{k})$ for some $k \in \N$, and so $F(A) \cap U_{\sigma|_{k}} = \emptyset$. This implies that $p_{\sigma} \notin F(A)$ for every $A \in \II$.
\end{proof}

In the case of separable complete metric spaces (\emph{Polish spaces}) we can give (apparently) simpler equivalent conditions than those of Theorem \ref{TEOR:treeWeight}.

 \begin{CORO}
 \label{CORO:baireCompact}
Let $\II$ be an ideal on $\N$. The following assertions are equivalent:
 \begin{enumerate}
 \item[$(1)$] $\II$ is a Baire ideal for every Polish space.
 \item[$(2')$] $\II$ is a Baire ideal for $2^{\N}$.
 \item[$(3')$] let $f\colon 2^{< \N} \rto \mathcal{P}(\II)$ be monotonic and such that for every $s \in 2^{< \N}$:
 
\noindent (i) $f(s)$ is hereditary, (ii) $\bigcup_{t \succcurlyeq s}{f(t)} = \II$.

\noindent Then, there is $\sigma \in 2^{\N}$ such that $\bigcup_{k \in \N}{f(\sigma|_{k})} = \II$.  
 \end{enumerate}
 \end{CORO} 

\begin{proof}
$(1) \Rightarrow (2')$ is clear and $(2') \Rightarrow (3')$ follows as in the proof of $(2) \Rightarrow (3)$ in Theorem \ref{TEOR:treeWeight}.

\noindent $(3') \Rightarrow (1)$: We will show that (3) of Theorem \ref{TEOR:treeWeight} is satisfied for the discrete topological space $D= \N_{0} := \N \cup \{ 0\}$.  Let $f\colon \N_{0}^{< \N} \rto \mathcal{P}(\II)$ be a monotonic function such that for every $s \in \N_{0}^{< \N}$ we have that $f(s)$ is hereditary and $\bigcup_{t  \succcurlyeq s}{f(t)} = \II$. We establishes now a monotonic onto map $\psi\colon 2^{< \N} \rto \N_{0}^{<\N}$. Each $s \in 2^{< \N}$ can be seen as a sequence of blocks of $0$'s separated by $1$, so it can be written as $s=(0^{n_{1}},1,0^{n_{2}},1, \ldots, 1, 0^{n_{k+1}})$ where $0^0 := \emptyset$ and $0^k := (0,...,0)$ ($k$ times) if $k>0$. Put $\psi(s) = (n_{1},\ldots,n_{k})$ with the convention $\psi((0^{k})) = \emptyset$ for each $k \geq 0$. It satisfies the conditions above, so the map $g\colon 2^{< \N} \rto \mathcal{P}(\II)$ given by $g(s) = f(\psi(s))$ is monotonic and clearly satisfies (i) and (ii). By hypothesis there is $\sigma \in 2^{\N}$ such
  that $\bigcup_{k \in \N}{g(\sigma|_{k})} = \II$. There are two possibilities: If there is $k_{0} \in \N$ such that $\sigma(k) = 0$ whenever $k \geq k_{0}$ then $\psi(\sigma|_{k}) = \psi(\sigma|_{k_{0}})$ for $ k \geq k_{0}$ and we conclude that $g(\sigma|_{k_{0}}) = \II$. On the other hand, if the support of $\sigma$ is not finite, we can write $\sigma = (0^{n_{1}},1,0^{n_{2}},1, \ldots)$ so $\tau = (n_{1},..., n_{k},...) \in \N_{0}^{\N}$ satisfies
\[ \bigcup_{k \in \N}{f(\tau|_{k})} = \bigcup_{k \in \N}{g((0^{n_{1}},1,0^{n_{2}},1, \ldots, 0^{n_{k}}, 1))} = \II. \]
\end{proof}
  
\begin{REM}
Since $2^{\N}$ is a compact space, the statements (1), (2') and (3') of Corollary \ref{CORO:baireCompact} are also equivalent to ``$\II$ is a Baire ideal for every compact metric space''. Recall that the dual filters of these ideals were introduced in \cite[Definition 4.2]{SchurFilters} under the name of\emph{category respecting filters}, as an example of filters having the \emph{Schur's property}: every weakly $\FF$-convergent to $0$ sequence in $\ell^{1}$ is $\FF$-convergent in norm to zero.
\end{REM}  
  
Looking at Theorem \ref{TEOR:treeWeight} it is natural to ask whether there is a complete metric space $X_{0}$ such that $\II$ is a Baire ideal for every complete metric space $X$ (independently of the weight of the space) whenever it is a Baire ideal for $X_{0}$. The last part of this section is devoted to find such a space.

\begin{PROP}\label{PROP:limitingCardinal}
Let $X$ be a complete metric space with $|X| > \mathfrak{c}$ and $\{D_{\alpha}\}_{\alpha \in {\mathfrak{c}}}$ a family of non-empty closed elements of $\NWD(X)$. Then, there exists a closed subset $\Omega \subseteq X$ with $|\Omega| \leq \mathfrak{c}$ and such that
$D_{\alpha} \cap \Omega$ is non-empty and belongs to $\NWD(\Omega)$.
\end{PROP}

\begin{proof}
We start by constructing an increasing sequence $(\Omega_{n})_{n \geq 0}$ of subsets of $X$ in the following way. Take $\Omega_{0}$ an arbitrary set with $|\Omega_{0}| \leq \mathfrak{c}$ and $\Omega_{0} \cap D_{\alpha} \neq \emptyset$ for every $\alpha \in \mathfrak{c}$. Now suppose that $\Omega_{n}$ has been constructed. Using that each $D_{\alpha}$ is nowhere dense in $X$, for each $\alpha \in \mathfrak{c}$ and every $p \in
\Omega_{n} \cap D_{\alpha}$ we can find a sequence $(y^{p,
\alpha}_{k})_{k \in \N}$ in $X \setminus D_{\alpha}$ which converges to $p$. Then define
\[ \Omega_{n+1} := \Omega_{n} \cup \{ y_{k}^{p, \alpha}\colon p \in D_{\alpha}
\cap \Omega_{n}, \alpha \in \mathfrak{c}, k \in \N \}. \]
Define $\Omega$ as the closure of $ \bigcup_{n \in \N}{\Omega_{n}}$ in $X$. Notice that $\Omega$ has cardinality at most the continuum since each of its elements is the limit of a sequence in $\bigcup_{n \in \N}{\Omega_{n}}$, which has cardinality less or equal than $\mathfrak{c}$. To finish the proof we have to check that $D_{\alpha} \cap \Omega \in \NWD(\Omega)$ for every $\alpha \in \mathfrak{c}$. Fix such an $\alpha$ and let $U$ be an open subset of $X$ with $U \cap \Omega \neq \emptyset$. We can find $n_{0} \in \N$ with $U \cap \Omega_{n_{0}} \neq \emptyset$, and by construction there exists $y_{k}^{p, \alpha} \in U \cap \Omega_{n_{0}+1} \setminus D_{\alpha}$, so $U \cap \Omega$ cannot be contained in $D_{\alpha} \cap \Omega$.
\end{proof}

 \begin{CORO}
 Let $\II$ be an ideal on $\N$. Then, $\II$ is a Baire ideal (for every complete metric space) whenever it is a Baire ideal for $D^{\N}$ where $D$ is a discrete space with $|D| = \mathfrak{c}$.
 \end{CORO}

\begin{proof}
Let $X$ be a complete metric space. If $|X| \leq \mathfrak{c}$ then $\II$ is a Baire ideal for $X$ by Theorem \ref{TEOR:treeWeight}. We prove now that this is also true if $X$ has cardinality strictly bigger than the continuum. Given a monotonic function $F\colon \II \rto \NWD(X)$ we can assume that its images are closed sets by taking the closure. Proposition \ref{PROP:limitingCardinal} provides a closed subset $\Omega$ of $X$ of cardinality at most $\mathfrak{c}$ such that the map $G\colon \II \rto \NWD(\Omega)$ given by $G(A) = F(A) \cap \Omega$ is well-defined and monotonic. Therefore $\Omega \neq \bigcup{\{ G(A)\colon A \in \II \}}$ since $\II$ is a Baire ideal for $\Omega$, which implies that $X \neq \bigcup{\{ F(A)\colon A \in \II \}}$. 
\end{proof}


\section{Baire ideals in some classes of ideals}

\subsection{Analytic ideals}

The ideals we are dealing with in this section are the mostly used one (see \cite{Farah}, \cite{SoleckiAnalytic}). Roughly speaking, all the ideals that can be defined by explicit formulas are analytic. Due to one of equivalent definitions, see \cite[Theorem 7.9, Definition 14.1]{kechris-descriptive}, a subset $A$ of a Polish space $X$ is \emph{analytic} if it can be represented as an image of a continuous map acting from $\N^{\N}$ to $X$. In particular, all Borel sets are analytic. An ideal $\II$ is \emph{analytic} if it forms an analytic subset of the Polish topological space $2^{\N}$.

\begin{PROP}
If $\II$ is an analytic ideal on $\N$ which is not countably generated, then it is not a Baire ideal for some Polish space.
\end{PROP}

\begin{proof}
Suppose that there exists a continuous function $g\colon \N^{\N} \rto 2^{\N}$ whose image is the ideal $\II$. Following the notation of section~\ref{sec:characterizations}, for each $s \in \N^{<\N}$ we write
\[ U_{s} = \{ \sigma \in \N^{\N}\colon \sigma \succcurlyeq s \} \mbox{ \hspace{2mm} and \hspace{2mm} } [s]^{g}:= \{ g(\sigma)\colon \sigma \in \N^{\N}, \sigma \succcurlyeq s \}.\]
We consider now the following set
\[ X:= \N^{\N} \setminus \bigcup{\{ U_{s}\colon [s]^{g} \mbox{ is contained in a countably generated subideal of $\II$} \}}. \]
It is non-empty since $[\emptyset]^{g} = \II$ is not countably generated. Moreover, it is a closed subset of $\N^{\N}$ since the sets $U_{s}$ are open, so $X$ is a Polish space. Consider the map $ F\colon \II \rto \NWD(X)$ defined for each $A \in \II$ as
\[ F(A)= \{ \sigma \in X\colon g(\sigma) \subset A \}.\]
Let us see that for every $A \in \II$ we have that $F(A) \in \NWD(X)$. Let $U_{s} \cap X$ be a (relatively) open set in $X$. If $U_{s} \cap X \neq \emptyset$ then $[s]^{g}$ is not countably generated, which in particular implies that there is $\sigma \in \N^{\N}, \sigma \succcurlyeq s$ such that $g(\sigma) \nsubseteq A$. Therefore there is $m \in \N$ with $g(\sigma) \cap \{ 1,\ldots, m\} \nsubseteq A$ and by continuity we can find $\sigma \succcurlyeq t \succcurlyeq s$ with $g(\tau) \cap \{ 1, \ldots,m\} = g(\sigma) \cap \{ 1,\ldots,m\}$ for every $\tau \succcurlyeq t$, which means that $\sigma \in U_{t} \subset U_{s} \setminus F(A)$. Hence, $F$ is a well-defined monotonic function. On the other hand, $X = \bigcup{\{ F(A)\colon A \in \II \}}$ since $\sigma \in F(g(\sigma))$ for every $\sigma \in X$.
\end{proof}

\subsection{Generalized uncountably generated P-ideals}

An ideal $\II$ on $\N$ is said to be a \emph{P-ideal} if for every sequence $(A_{n})_{n \in \N}$ in $\II$ there exists $B \in \II$ such that $A_{n} \setminus B \in \Fin$ for each $n \in \N$. Moreover, such an ideal $\II$ is not countably generated if and only if for every $A \in \II$ there exists $B \in \II$ satisfying $B \setminus A\notin \Fin$. This motivates the following definition. 

\begin{DEFI}
\label{DEFI:GeneralPideal} Let $\II, \II_{0}$ be ideals on
$\N$. We say that $\II$ is a \emph{P($\II_{0}$)-ideal} if it satisfies the following conditions:
\begin{enumerate}
\item[(I)] Given $(A_{n})_{n \in \N}$ in $\II$ there is $A
\in \II$ such that $A_{n} \setminus A \in \II_{0}$ for each $n \in
\N$.
\item[(II)] For each $A \in \II$ there exists $B \in \II$ such that $B
\setminus A \notin \II_{0}$.
\end{enumerate}
\end{DEFI}
\noindent We show now how to construct P($\II_{0}$)-ideals which are not P($\Fin$)-ideals.
\begin{EJEM}
Let $\II, \JJ$ be ideals on $\N$ where $\II$ is an uncountably generated P-ideal. Take the direct sum $\II \oplus \JJ$,i.e. the family of all subsets $A \subseteq \N \times \{ 0,1\} \cong \N$ such that
\[ A^{0}=\{ n \in \N\colon (n,0) \in A \} \in \II \mbox{ \hspace{2mm} and \hspace{2mm} }  A^{1}=\{ n \in \N\colon (n,1) \in A \} \in \JJ, \]
and write $\II' := \II \oplus \Fin$, $\JJ' := \Fin \oplus \JJ$. We claim that $\II \oplus \JJ$ is a P($\JJ'$)-ideal. To see \emph{(I)}, let $(A_{n})_{n \in \N}$ be a sequence of elements in $\II \oplus \JJ$. Since $\II$ is a P-ideal, there exists $B \in \II$ such that $A_{n}^{0} \setminus B \in \Fin$ for every $n \in \N$, so 
\[ A_{n} \setminus (B \times \{0\}) \subseteq ((A_{n}^{0} \setminus B) \times \{ 0\}) \cup (A_{n}^{1} \times \{ 1\}) \in \JJ'.\]
Since $\II$ is not countably generated, given $A \in \II \oplus \JJ$ we can find $B \in \II$ satisfying $B \setminus A^{0} \notin \Fin$, so $(B \times \{ 0\}) \setminus A \notin \JJ'$. This proves \emph{(II)}. 

Observe that if $\II \oplus \JJ$ is a P($\Fin$)-ideal then $\JJ$ is a P-ideal, since given a family $(J_{n})_{n \in \N}$ in $\JJ$ there exists $A \in \II \oplus \JJ$ such that $(J_{n} \times \{ 1\}) \setminus A \in \Fin$ and hence $J_{n} \setminus A^{1} \in \Fin$ for every $n \in \N$. In particular taking $\JJ$ not being a P-ideal we get that $\II \oplus \JJ$ is a P($\JJ'$)-ideal but not a P($\Fin$)-ideal. 
\end{EJEM}

Recall that $(\ell^{\infty}, \lVert \cdot \lVert_{\infty})$ is the Banach space of all bounded real sequences. If $\II_{0}$ is an ideal on $\N$, we write 
\[ c_{0}(\II_{0}) = \left\{ (x_{n})_{n \in \N} \in \ell^{\infty} \colon \{ n \in \N\colon |x_{n}| > \varepsilon \} \in \II_{0} \mbox{ for every $\varepsilon > 0$} \right\}. \]
It is easy to check that it is a closed subspace of $\ell^{\infty}$. If $(x_{n})_{n \in \N}, (y_{n})_{n \in \N}$ are sequences of real numbers, we will write $(x_{n})_{n \in \N} \leq (y_{n})_{n \in \N}$ if $x_{n} \leq y_{n}$ for each $n \in \N$. If $A \subset \N$, we will denote by $\chi_{A}$ the characteristic function of $A$. 

\begin{PROP}
\label{PROP:PidealsNotBideals}
Let $\II, \II_{0}$ be ideals on $\N$ such that $\II$ is a P($\II_{0}$)-ideal. Then $\II$ is not a Baire ideal.
\end{PROP}

\begin{proof}
We will find a closed subspace $X$ of $E := \ell^{\infty}/c_{0}(\II_{0})$ for which $\II$ is not a Baire ideal. For each $A \in \II$ consider $F_{A} = \{ [(x_{n})_{n \in \N}] \colon 0 \leq (x_{n})_{n \in \N} \leq \chi_{A} \}$. Notice that $[(x_{n})_{n \in \N}] \in F_{A}$ if and only if there is $(y_{n})_{n \in \N} \in c_{0}(\II_{0})$ such that $0 \leq (x_{n} + y_{n})_{n \in \N} \leq \chi_{A}$, or equivalently, if for every $\varepsilon > 0$ we have that
\[ \{ n \in \N \setminus A\colon |x_{n}| > \varepsilon\} \cup \{ n \in A\colon x_{n}> 1+ \varepsilon \} \cup \{ n \in A\colon x_{n} < -\varepsilon \} \in \II_{0}. \]
From this it easily follows that each $F_{A}$ is closed. We can also deduce:
\begin{enumerate}
\item \label{item1} If $A, B \in \II$ then $A \setminus B \in \II_{0}$ if and only if $F_{A} \subseteq F_{B}$.

\noindent Suppose that $A \setminus B \in \II_{0}$ and let $(x_{n})_{n \in \N}$ be a sequence with $0 \leq (x_{n})_{n \in \N} \leq \chi_{A}$. Define $(y_{n})_{n \in \N}$ as $y_{n} = 0$ if $n \notin A \setminus B$ and $y_{n} = - x_{n}$ if $n \in A \setminus B$. It is clear that $(y_{n})_{n \in \N} \in c_{0}(\II_{0})$ and that $ [(x_{n})_{n \in \N}] = [(x_{n})_{n \in \N} + (y_{n})_{n \in \N}] \in F_{B}$. Conversely, if $F_{A} \subseteq F_{B}$ then $[\chi_{A}] \in F_{B}$ and the condition above implies that $A \setminus B \subset \{ n \in \N \setminus B\colon \chi_{A}(n) > 1/2  \} \in \II_{0}$.

\item \label{item2} For every $(A_{n})_{n \in \N} \subseteq \II$ there exists $B \in \II$ such that
$ \bigcup_{n \in \N}{F_{A_{n}}} \subset F_{B}$. 

\noindent This follows from Definition \ref{DEFI:GeneralPideal} (I) and the previous property.

\item \label{item3} If $A \in \II$ there is $B \in \II$ such that $F_{A}$ is a nowhere dense subset of $F_{B}$.

\noindent By Definition \ref{DEFI:GeneralPideal} (II), there exists $B \in \II$ such that $B \setminus A \notin \II_{0}$. Replacing $B$ by $A \cup B$, we can assume that $A \subset B$ so that $ F_{A} \subset F_{B}$. To see that $F_{A}$ is nowhere dense in $F_{B}$, take a sequence $0 \leq (x_{n})_{n \in \N} \leq \chi_{A}$ and $V$ an open set in $E$ containing $[(x_{n})_{n \in \N}]$. Fixed $0< \delta < 1$, the sequence $(y_{n})_{n \in \N} = (x_{n})_{n \in \N} + \delta \chi_{B \setminus A}$ satisfies that $\| [(x_{n})_{n \in \N}] - [(y_{n})_{n \in \N}] \| \leq \delta$. We can take $\delta$ small enough so that $[(y_{n})_{n \in \N}] \in V$. On the other hand, $[(y_{n})_{n \in \N}] \notin F_{A}$ since 
$B \setminus A \subset \{ n \in \N \setminus A\colon |y_{n}| > \delta/2 \}$, but $[(y_{n})_{n \in \N}] \in F_{B}$ because $0 \leq (y_{n})_{n \in \N} \leq \chi_{B}$. To summarize, we have proved that $ [(y_{n})_{n \in \N} ] \in (V \cap F_{B}) \setminus F_{A}$.
\end{enumerate}

Finally, the set $X := \bigcup{\{ F_{A}\colon A \in \II\}}$ is a closed subspace of $E$ (by (\ref{item2})) that can be covered by a family of nowhere dense subsets $\{ F_{A}\colon A \in \II\} \subset \NWD(X)$ (by (\ref{item3})) with the property that $F_{A} \subset F_{B}$ whenever $A \subset B \in \II$ (by (\ref{item1})). 
\end{proof}

Following \cite{Farah}, the \emph{orthogonal} of an ideal $\II$ on $\N$ is defined as the set of all $A \in \mathcal{P}(\N)$ such that $A \cap B \in \Fin$ for every $B \in \II$. The ideal $\II$ is said to be \emph{tall} if its orthogonal is equal to $\Fin$, or equivalently, if given $A \notin \Fin$ there is $B \in \II \setminus \Fin$ with $B \subset A$. We now extend these definitions.

\begin{DEFI}
Let $\II, \II_{0}$ be ideals on $\N$. The \emph{orthogonal} of $\II$ respect to $\II_{0}$ is defined as $ \II^{\perp \II_{0}} := \left\{ A \subset \N\colon A \cap B \in \II_{0} \text{ for all $B \in \II$} \right\}$. We say that $\II$ is \emph{$\II_{0}$-tall} if $\II^{\perp \II_{0}} = \II_{0}$, i.e. if for every $A \notin \II_{0}$ there exists $B \in \II\setminus \II_{0}$ such that $B \subset A$.
\end{DEFI}  

\begin{PROP} \label{prop-majorated}
Suppose that $\II$ is a \emph{P($\II_{0}$)}-ideal which is $\II_{0}$-tall. Then every ideal $\JJ$ containing $\II \cup \II_{0}$ is not a Baire ideal.
\end{PROP}

\begin{proof}
Following the notation of the proof of Proposition \ref{PROP:PidealsNotBideals} consider the family $\{ F_{C}\colon C \in \II\}$ of closed subsets of $E = \ell^{\infty}/c_{0}(\II_{0})$ where $F_{C} = \{ [(x_{n})_{n \in \N}] \colon 0 \leq (x_{n})_{n \in \N} \leq \chi_{C} \}$. We define a function $f\colon \JJ \rto \mathcal{P}(E)$ that assigns to each $A \in \JJ$ the element
\[ f(A) = \bigcup{\{ F_{C}\colon C \in \II, C \subset A \}}. \]
Since $\bigcup{\{ f(A)\colon A \in \JJ \}} = \bigcup{\{ F_{C}\colon C \in \II \}} = X$ was a complete metric space, we just have to check that $f(A) \in \NWD(X)$ for every $A \in \JJ$ to finish the proof. Fix $A \in \JJ$. Since $\N \setminus A \notin \JJ \supset \II_{0}$, there exists by hypothesis $B \in \II \setminus \II_{0}$ with $B \subset \N \setminus A$. Now we can repeat the argument of (\ref{item3}) in the proof of Proposition \ref{PROP:PidealsNotBideals} to deduce that $f(A)$ is a nowhere dense subset of $f(A \cup B)$, and hence of $X$.
\end{proof}

\begin{CORO} \label{majorated}
If $\II$ is a tall P-ideal on $\N$ then every ideal $\JJ \supseteq \II$ is not a Baire ideal.
\end{CORO}

\subsection{Ideals generated by AD families}

Recall that a family of sets $\AA$ is said to be  \emph{almost disjoint} (AD) if the intersection of any pair of its elements is finite. We give now examples of ideals which do not belong to the class of P($\II_{0}$)-ideals for any ideal $\II_{0}$ on $\N$.

\begin{PROP}
Let $\II$ be an ideal on $\N$ generated by an infinite family $\AA$ 
of almost disjoint infinite sets. Then $\II$ is not a 
\emph{P($\II_{0}$)}-ideal for any ideal $\II_{0}$ on $\N$.
\end{PROP}

\begin{proof}
Suppose that $\II$ is a P($\II_{0}$)-ideal for some ideal $\II_{0}$ on $\N$. Fix $A \in \AA$ and take a sequence $(B_{n})_{n \in \N}$ in $\AA \setminus \{ A\}$ whose elements are all different. Now define inductively $A_{0}:=A$ and $A_{n+1}:=A_{n} \cup B_{n}$ for every $n \in \N$. The assumption implies the existence of some $C = C_{1} \cup ... \cup C_{k}$ ($C_{i} \in \AA$ for every $i$) such that $A_{n} \setminus C \in \II_{0}$ for every $n \geq 0$. But there must be some $A_{n_{0}}$ different from the elements $C_{i}$'s. Therefore
\[ A_{n_{0}} \setminus C = A_{n_{0}} \setminus \left((A_{n_{0}} \cap C_{1}) \cup ... \cup (A_{n_{0}} \cap C_{k}) \right) \in \II_{0} \]
and hence $A_{n_{0}} \in \II_{0}$ since $A_{n_{0}} \cap C_{i}$ is finite for each $i=1,\ldots, k$. But this means that $\AA \subset \II_{0}$ and hence $\II \subset \II_{0}$, contradicting Definition \ref{DEFI:GeneralPideal} (II).
\end{proof}

By a result of Mathias \cite{HappyMathias}, ideals $\II$ generated by a maximal almost disjoint (MAD) family of subsets of $\N$ are not analytic, so they do not belong to any of the classes of ideals considered in the previous subsections. However, ideals generated by an AD family $\AA$ are not Baire ideals when the generating family $\AA$ has cardinality $\mathfrak{c}$, as it follows from the next result.

\begin{PROP}
\label{PROP:almostDisjointFamily}
Let $\II$ be an ideal on $\N$. Suppose that there exists a family $\AA \subseteq \II$ with $|\AA| = \mathfrak{c}$ and such that every $B \in \II$ satisfies $ | \{ A \in \AA\colon A \subseteq B \} | < + \infty$. Then $\II$ is not a Baire ideal for $2^{\N}$.
\end{PROP}

\begin{proof}
We will use Corollary \ref{CORO:baireCompact}. Since $\AA$ has cardinality $\mathfrak{c}$ we can index its elements as $\AA = \{ A_{\sigma}\colon \sigma \in 2^{\N} \}$ and define a function $f\colon 2^{<\N} \rto \mathcal{P}(\II)$ as
\[ f(s) = \left\{ B \in \II\colon A_{\sigma} \nsubseteq B \mbox{ whenever } s \preccurlyeq \sigma \in 2^{\N} \right\} \hspace{2mm} \mbox{ for } s \in 2^{< \N}.\]
It is clearly monotonic. Fix any $s \in 2^{< \N}$. Obviously $f(s)$ is hereditary. Given $B \in \II$, since $\{ \sigma \in 2^\N: A_{\sigma} \subset B\}$ is finite, we can find $t \succcurlyeq s$ such that $A_{\sigma} \nsubseteq B$ for every $\sigma \succcurlyeq t$, so that $B \in f(t)$. This shows that $\bigcup_{t \succcurlyeq s}{f(t)} = \II$. Finally, given any $\sigma \in 2^{\N}$ we have that $A_{\sigma} \notin f(\sigma|_{k})$ for every $k \in \N$.
\end{proof}

\section{An uncountably generated Baire ideal for every Polish space}
\label{sec:example}

The aim of this section is to point out that there are models of ZFC in which we can find P-ideals $\II$ with character equal to $\mathfrak{c}$ which are Baire ideals for every Polish space, despite of the fact that they are not Baire ideals for some (non-separable) complete metric space $X_{0}$ (see Proposition \ref{PROP:PidealsNotBideals}).

Let $M[G]$ be a model of set theory obtained by adding $\aleph_2$ many Cohen reals $\{x_\alpha : \alpha<\omega_2\}\subset 2^\mathbb{N}$ to a ground model $M$ where Martin's Axiom and $\mathfrak{c}=\aleph_2$ hold. In the model $M$ there exists an $\omega_2$-chain in $\mathcal{P}(\N)/\Fin$. That is, we have a family $\{C_\alpha\}_{\alpha<\omega_2}$ of subsets of $\mathbb{N}$ such that $C_\alpha \subset^\ast C_\beta$ whenever $\alpha<\beta$; or equivalently, $C_\alpha\setminus C_\beta$ is finite, but $C_\beta\setminus C_\alpha$ is infinite whenever $\alpha<\beta$. We define, in the model $M[G]$,  $$\mathcal{I} = \{a\subset \mathbb{N} : \exists \alpha : a\subset^\ast C_\alpha\}$$

It is clear that $\II$ cannot be generated by less than $\aleph_2$ elements, and $\mathfrak{c}=\aleph_2$ in $M[G]$. We are going to prove that $\II$ satisfies property $(\Box)$ for the case $X = 2^{\N}$. This in particular will imply that $\II$ is a Baire ideal for $2^{\N}$, and by Corollary \ref{CORO:baireCompact} this is equivalent to say that $\II$ is a Baire ideal for every Polish space. 

We will reason by contradiction assuming that we can write $2^\N = \bigcup_{A\in\mathcal{I}} X_A$  in such a way that each $X_A$ is meager and $A\subset B$ implies $X_A\subset X_B$. If we consider $Y_A = \bigcup\{X_{A\cup F} : F \subset \N \text{ is finite}\}$, we still have that each $Y_A$ is meager, $2^\N = \bigcup_{A\in \mathcal{I}} Y_A$ and now $A\subset^\ast B$ implies $Y_A\subset Y_B$. Given the definition of $\mathcal{I}$, this implies that $2^\N = \bigcup_{\alpha<\omega_2}Y_{C_\alpha}$ and $Y_{C_\alpha}\subset Y_{C_\beta}$ if $\alpha<\beta$. 
In particular, there must exist $\alpha$ such that $\{x_\gamma : \gamma<\omega_1\}\subset Y_{C_\alpha}$, and in particular, $\{x_\gamma : \gamma<\omega_1\}$ would be a meager set. However, this contradicts the following general fact about Cohen reals, whose proof we include for the reader's convenience:

\begin{PROP} If $Z\subset \omega_1$ is an uncountable set, then $\{x_\alpha : \alpha\in Z\}$ fails to be nowhere dense in $2^\N$.
\end{PROP}

\begin{proof}
Let $\mathbb{P}$ be the forcing that adds the Cohen reals $\{x_\alpha : \alpha<\omega_2\}$, and let $\{\dot{x}_\alpha : \alpha<\omega_2\}$ be the canonical names of these Cohen reals. Remember, elements $p\in \mathbb{P}$ are $\{0,1\}$-valued functions whose domain is a finite subset of $\mathbb{N}\times\omega_2$, endowed with the natural extension order that $p<q$ iff $\dom{(p)}\supset \dom{(q)}$ and $p|_{\dom{(q)}} = q$.  For each  $s\in 2^{<\mathbb{N}}$ we can consider the clopen set $[s]$ of $2^\mathbb{N}$ consisting of all infinite sequence that end-extend $s$. This describes a basis of open subsets of $2^\mathbb{N}$. Moreover $[t]\subset [s]$ if and only if $t$ is an end-extension of $s$, that we write as $s\prec t$. So, if $\{x_\alpha : \alpha\in Z\}$ is nowhere dense, then there exists a funciton $t:2^{<\mathbb{N}}\rto 2^{<\mathbb{N}}$ such that, for all $s\in2^{<\mathbb{N}}$:
\begin{enumerate}
\item $s\prec t(s)$,
\item $[t(s)]\cap \{x_\alpha : \alpha\in Z\} = \emptyset$
\end{enumerate} 

Let $\dot{t}$ be a name for $t$. For every $s\in 2^{<\mathbb{N}}$ there must exist a condition $p_s\in \mathbb{P}$ such that $p_s\Vdash \dot{t}(s) = t(s)$. Since $2^{<\mathbb{N}}$ is countable and each $p_s$ has finite domain, we can find $\alpha\in Z$ such that $(\mathbb{N}\times \{\alpha\}) \cap \dom{(p_s)} = \emptyset$ for all $s\in 2^{<\mathbb{N}}$. We know, by the definition of the function $t$, that $x_\alpha\not\in [t(s)]$ for any $s\in 2^{<\mathbb{N}}$. However, we are going to prove that for every condition $q$ in the generic filter of $\mathbb{P}$ there exists $q'<q$ such that
$$q'\Vdash \exists s\in 2^{<\mathbb{N}} : \dot{x}_\alpha \in [\dot{t}(s)]$$
and this is a contradiction. So fix $q\in G$. We can assume that the domain of $q$ is of the form $\{0,\ldots,n_0\}\times F_q$ with $\alpha\in F_q$. Consider $s\in 2^{<\mathbb{N}}$ given by $s(n) = q(n,\alpha)$ for $n\leq n_0$. We define the desired condition $q'$ as $q'(n,\gamma) = p_s(n,\gamma)$ if $(n,\gamma)\in \dom{(p_s)}$, and $q'(n,\alpha) = t(s)(n)$ for any $n$ where $t(s)$ is defined. This is possible by the choice we made of $\alpha$. This ensures, on the one hand, that $q'\leq p_s$, so $q'\Vdash \dot{t}(s) = t(s)$. On the other hand, by the way $q'$ is defined on $\mathbb{N}\times \{\alpha\}$, $q'\Vdash \dot{x}_\alpha \in [t(s)]$. Thus $q'\Vdash \dot{x}_\alpha\in [\dot{t}(s)]$ as desired. 
\end{proof}

\section{Coverings with compact subsets}
\label{sec:sompact}

In this section we present an example which shows that an analogue of Baire Theorem for coverings of $\N^{\N}$ by compact subsets can be valid for some uncountably generated analytic ideals. For an arbitrary topological space $X$, we write $\KK(X) := \{ A \subseteq X\colon \text{ $A$ is compact} \}$. Given a partially ordered set $(D,\leq)$, we say that a subset $A \subseteq D$ is \emph{weakly bounded} if each infinite subset of $A$ contains an infinite bounded subset.

\begin{PROP}
Let $f\colon \II \rto \KK(X)$ be a monotonic function where $\II$ is an ideal on $\N$. If $\BB \subset \II$ is a weakly bounded subset then $\bigcup{\{ f(B)\colon B \in \BB\}}$ is countably compact.
\end{PROP}

\begin{proof}
Let $(x_{n})_{n \in \N}$ be an infinite sequence in $\bigcup{\{ f(B)\colon B \in \BB\}}$. We can choose for each $n \in \N$ an element $B_{n} \in \BB$ so that $x_{n} \in f(B_{n})$. Either if the range of $(B_{n})_{n \in \N}$ is finite or not, we can find by the hypothesis a subsequence $(B_{n_{k}})_{k \in \N}$ bounded by a set $A \in \II$. The corresponding subsequence $(b_{n_{k}})_{k \in \N}$ is contained in $f(A)$ by monotonicity, so it has a cluster point in $f(A)$ which will also be a cluster point of the original sequence $(b_{n})_{n \in \N}$.
\end{proof}

\begin{CORO}\label{CORO:sigmaCountablyBounded}
Let $\II$ be an ideal on $\N$ that can be written as a countable union of weakly bounded subsets. Then, every monotonic function $f\colon \II \rto \KK(\N^{\N})$ satisfy that $\bigcup{\{ f(A)\colon A \in \II \}} \neq \N^{\N}$.
\end{CORO}

\begin{proof}
Suppose that $\II = \bigcup_{n \in \N}{\BB_{n}}$ where $\BB_{n}$ is a weakly bounded subset of $\II$. Then $\bigcup{\{ f(B)\colon B \in \BB_{n} \}}$ is a compact subset of $\N^{\N}$, and hence $\bigcup{\{ f(A)\colon A \in \II \}}$ is a countable union of compact sets. But $\N^{\N}$ is not a countable union of compact sets as an easy application of Baire Category Theorem.
\end{proof}

We show now an example, taken from \cite[p. 178, Example 1]{Veli}, of an ideal that satisfies the hypothesis of Corollary \ref{CORO:sigmaCountablyBounded}. Consider the map $\varphi\colon \mathcal{P}(\N) \rto [0,+\infty]$ defined by
\[ \varphi(A) = \inf{\{  c > 0\colon |A \cap \{1,\ldots,2^{n}\}| \leq n^{c} \mbox{ for each $n \geq 2$}\}} \]
with the convention $\varphi(A) = +\infty$ if the infimum does not exist. This is a lower-semicontinuous submeasure, so the ideal $\II_{p} = \{ A \subset \N\colon \varphi(A) < + \infty \}$ is an $F_{\sigma}$ ideal (see \cite[p. 7, Lemma 1.2.2]{Farah}). It is easy to check that $\II$ is not countably generated, and it is a countable union of weakly bounded subsets since for each $c > 0$
\[ \II_{p}(c) = \{ A \subset \N\colon \varphi(A) < c \} \]
is weakly bounded, see \cite[p. 178, Example 1]{Veli} for the details.

\section*{Acknowledgements}

We would like to express our gratitude to Christina Brech for her ideas in the construction of the model $M[G]$ in section \ref{sec:example}.




\begin{thebibliography}{10}
\bibitem{SchurFilters} Avilés, A., Cascales, B., Kadets, V. and Leonov, A.:
\newblock \emph{The Schur $\ell^{1}$ theorem for filters}.
Journal of Mathematical Physics, Analysis, Geometry 
(Zh. Mat. Fiz. Anal. Geom.) 3, no. 4, 383-398 (2007).

\bibitem{con-gan-kad} Connor, J., Ganichev, M., and Kadets, V.:
\newblock \emph{A characterization of Banach spaces with separable duals via weak statistical convergence}. Journal of Mathematical Analysis and Applications, 244(1), 251-261, (2000).

\bibitem{Farah} Farah,I.:
\newblock \emph{Analytic quotients: Theory of Liftings for Quotients
over Analytic Ideals on the Integers}. Mem. Am. Math. Soc. 702,
171 p. (2000).

\bibitem{fremlinMartin} Fremlin, D. H.:
\newblock \emph{Consequences of Martin's axiom}. Cambridge University Press (1984)

\bibitem{Gan-kad} Ganichev, M. and Kadets, V.:  
\newblock \emph{Filter Convergence in Banach Spaces and 
generalized Bases} / in Taras Banakh (editor) General Topology in 
Banach Spaces : NOVA Science Publishers, Huntington,
New York; pp. 61 - 69 (2001).

\bibitem{kechris-descriptive} Kechris, A. S.: 
\newblock \emph{Classical descriptive set theory} (Vol. 156). New York: Springer-Verlag (1995).

\bibitem{kelleyTopology} Kelley, J. L.: 
\newblock \emph{General topology}.
Volume 27 of Graduate texts in Mathematics, Springer-Verlag (1955).

\bibitem{kochanekFbases} Kochanek, T.:
\newblock \emph{$\mathcal F$-bases with brackets and with individual brackets in Banach spaces}. Stud. Math. , vol. 211, No. 3, p. 259-268 (2012)

\bibitem{Veli} Louveau, A. and Velickovic, B.:
\newblock \emph{Analytic ideals and cofinal types}.
Ann. Pure Appl. Logic \textbf{99}, No.1-3, 171-195 (1999).

\bibitem{Mazur} Mazur, K.:
\newblock \emph{$F_{\sigma}$ ideals and $\omega_{1} \omega_{1}^{\ast}$-gaps
in the Boolean algebra $\mathcal{P}(\omega) / I$}. Fundam. Math.
\textbf{138}, No.2, 103-111 (1991).

\bibitem{SoleckiAnalytic} Solecki, S.:
\newblock \emph{Analytic ideals and its applications}.
Annals of Pure and Applied Logic \textbf{99}, No.1-3, 51-72
(1999).

\bibitem{HappyMathias} Mathias, A. R.:
\newblock \emph{Happy families}. Annals of Mathematical logic, 12(1), 59-111 (1977).



\end{thebibliography}
\end{document}